\documentclass{article}
\usepackage[utf8]{inputenc}
\usepackage{amsthm}
\usepackage{amssymb}
\usepackage{amsmath}
\usepackage{thm-restate}
\usepackage{mathtools}
\usepackage{graphicx}
\usepackage{subcaption}
\usepackage{xcolor}
\graphicspath{{./images/}}

\title{Separation profiles of graphs of fractals}
\author{Valeriia Gladkova, Verna Shum}

\newtheorem{lemma}{Lemma}
\newtheorem{corollary}{Corollary}[lemma]

\newtheorem{prop}{Proposition}
\newenvironment{claim}[1]{\par\noindent\leftskip=0.7cm\textbf{Claim.}\space#1}{}
\newenvironment{claimproof}[1]{\par\noindent\leftskip=0.7cm\textit{Proof.}\space#1}

\begin{document}

\maketitle
\begin{abstract}
    We continue the exploration of the relationship between conformal dimension and the separation profile by computing the separation of families of spheres in hyperbolic graphs whose boundaries are standard Sierpi\'{n}ski carpets and Menger sponges. In all cases, we show that the separation of these spheres is $n^{\frac{d-1}{d}}$ for some $d$ which is strictly smaller than the conformal dimension, in contrast to the case of rank 1 symmetric spaces of dimension $\geq 3$. The value of $d$ obtained naturally corresponds to a previously known lower bound on the conformal dimension of the associated fractal.
\end{abstract}

\section{Introduction}
The separation profile -- originally considered by Lipton-Tarjan for planar graphs \cite{LipTar} and then studied more generally by Benjamini-Schramm-Tim\'{a}r \cite{BenSchTim} -- is a function which measures how well connected finite subgraphs of a given size can be. Formally, for a finite graph $G$, let $L(G)$ be the largest component of $G$ and define
$$ cut^{\epsilon}(G) = \min\{|S| : |L(G-S)| \leq \epsilon |G|\} $$
Then for an infinite graph $X$, and any $\epsilon\in(0,1)$, the $\epsilon$\textit{-separation profile} is defined as
$$ sep^{\epsilon}_X(n) = \sup\{cut^{\epsilon}(G) :  G \leq X, |G| \leq n \} $$

We will write $f(x) \preceq g(x)$ for $f(x) = \mathcal{O}(g(x))$ and $f(x) \asymp g(x)$ for $f(x) = \Theta (g(x))$ (in other words, $g(x) \preceq f(x) \preceq g(x)$). It is an easy exercise to show that $\epsilon$-separation profiles do not depend on the choice of $\epsilon\in(0,1)$ up to $\asymp$.

The modern motivation for studying separation comes from geometric group theory, since separation is a rare example of an invariant which behaves monotonically with respect to subgroups: given any finitely generated group $G$ and a finitely generated subgroup $H$ of $G$, we have $sep_H \preceq sep_G$. For hyperbolic groups, Hume-Mackay-Tessera prove that the separation profile can be bounded from above by $C_\epsilon n^{\frac{Q-1}{Q}+\epsilon}$ where $Q$ is the equivariant conformal dimension of the boundary of the group and $\epsilon$ is any positive real \cite{HMT_Poincareprofiles}, so lower bounds on the separation profile give lower bounds on the conformal dimension of the boundary. In the particular case of rank 1 symmetric spaces\footnote{except the real hyperbolic plane} they prove that the separation profile is exactly $n^{\frac{Q-1}{Q}}$ and the lower bound is attained both by metric balls, and by annuli of thickness $1$. The same lower bound can also be proved for any hyperbolic group whose boundary satisfies a $(1-1)$--Poincar\'{e} inequality \cite[Theorem 11.1]{HMT_Poincareprofiles}.

The goal of this paper is to prove that there are hyperbolic graphs whose boundaries have conformal dimension $Q>1$ where spheres do not have separation profile $n^{\frac{Q-1}{Q}}$. The examples we consider are hyperbolic cones in the sense of \cite{BonkSchramm} over the standard ``middle-thirds'' Sierpi\'{n}ski carpet $\mathcal S$ and the Menger sponge $\mathcal M$. These are interesting choices for two natural reasons:
\begin{itemize}
\item they do not satisfy a Poincar\'{e} inequality, see for instance \cite{Hein,HK}, \item in Gromov's density model a random group at density $0<d<\frac12$ is almost surely hyperbolic with boundary homeomorphic to the Menger sponge \cite{Gr_rand,DGP_split}.
\end{itemize}

For clarity, we first consider the case of the Sierpinski carpet.

Let $\mathcal S_0=[0,1]^2$. Each $\mathcal S_k$ is the union of $8^k$ squares of the form $[a3^{-k},(a+1)3^{-k}]\times [b3^{-k},(b+1)3^{-k}]$ with $a,b\in\{0,\ldots,3^k-1\}$. We obtain $\mathcal S_{k+1}$ by dividing each of these squares into 9 equal squares of side length $3^{-(k+1)}$ and removing the interior of the central one. We define $\mathcal S=\bigcap_{k\geq 0} \mathcal{S}_k$. 

For convenience we will write non-negative integers in base 3, so for each $x\in\mathbb N$ we define $x_i\in\{0,1,2\}$ so that $x=\sum_{i\geq 0} x_i3^i$. Now define $\Gamma_k$ to be the graph with vertex set
$$V_k=\{(x,y) \in \mathbb{Z}^2 \cap [0, 3^k)^{2}:\nexists \ i \text{ s.t. } x_i = y_i = 1\}$$
and edges $(x,y)(x',y')$ whenever $|x-x'|+|y-y'|=1$.

\begin{figure}[ht!]
\begin{subfigure}{0.5\textwidth}
\caption{$\Gamma_0$}
\includegraphics[width=\textwidth]{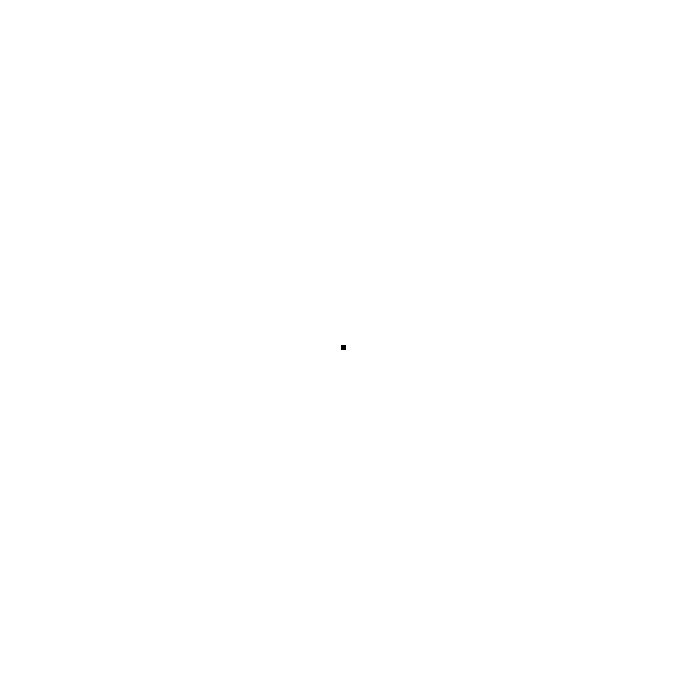}
\end{subfigure}
\begin{subfigure}{0.5\textwidth}
\caption{$\Gamma_1$}
\includegraphics[width=\textwidth]{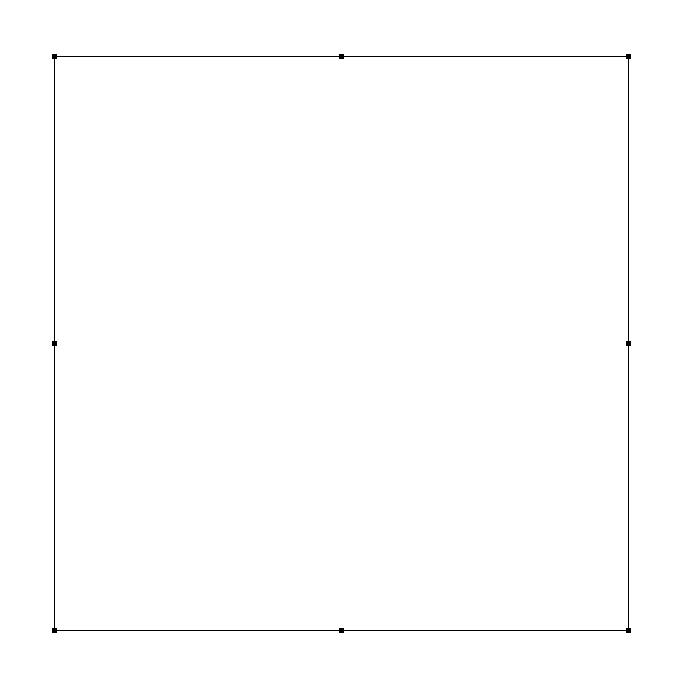}
\end{subfigure}
\begin{subfigure}{0.5\textwidth}
\caption{$\Gamma_2$}
\includegraphics[width=\textwidth]{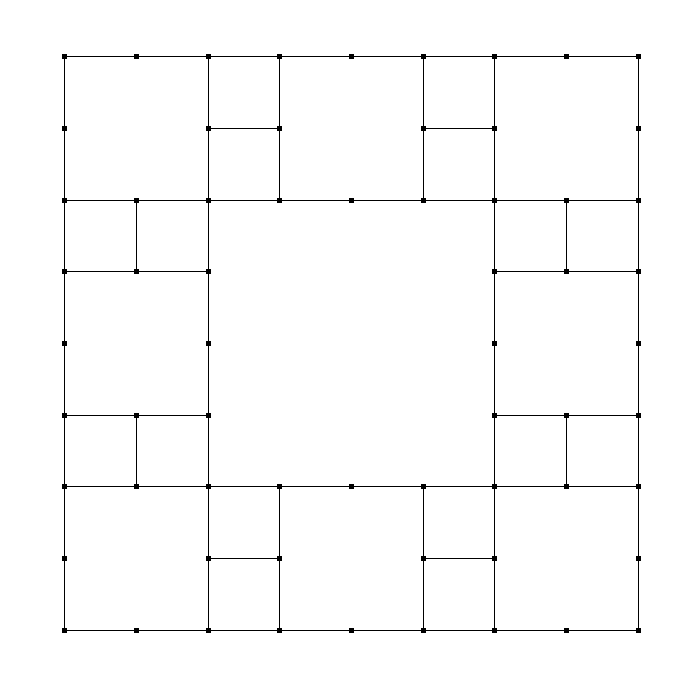}
\end{subfigure}
\begin{subfigure}{0.5\textwidth}
\caption{$\Gamma_3$}
\includegraphics[width=\textwidth]{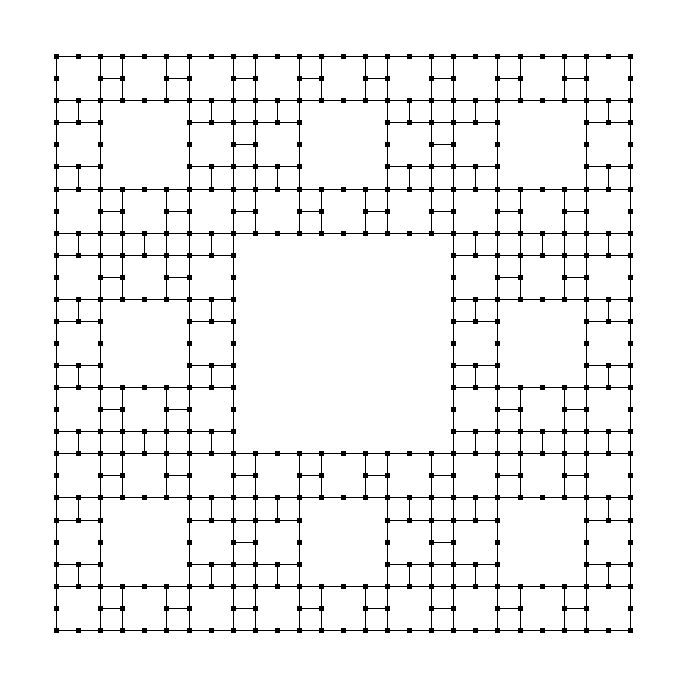}
\end{subfigure}
\caption{The first 4 iterations of the Sierpi\'{n}ski carpet}
\label{fig:its}
\end{figure}

We construct a hyperbolic graph $\Gamma$ from $\bigsqcup \Gamma_k$ by connecting each $(x,y)$ in $V_{k-1}$ to every $(x',y')\in V_k$ such that $x_i=x'_i$ and $y_i=y'_i$ for all $0\leq i < k-1$. We have $\partial_\infty\Gamma=\mathcal S$, and the spheres of $\Gamma$ centred at the unique vertex in $V_0$ are simply the graphs $\Gamma_k$. Note that the graphs $\Gamma_k$ are nested, so to study the separation of these spheres it suffices to work with $S=\bigcup_{k\geq 0} \Gamma_k$.

\begin{restatable}{theorem}{sierpinski}
\label{thm:sierpinski}
In each $\Gamma_k$ the ``subgraph of complete lines'' (cf.\ Figure \ref{fig:complete}) $C_k$ has $2\cdot 6^k-4^k$ vertices and $cut^{1/2}(C_k) \geq 2^{k-1}$. Moreover, this is optimal as $sep_S(r)\asymp n^{\log(2)/\log(6)}$.
\end{restatable}

The conformal dimension of a metric space $X$, originally introduced by Pansu \cite{Pansu_confdim}, is the infimal Hausdorff dimension of all spaces quasi--symmetric to $X$. The exact conformal dimension of the Sierpinski carpet is still unknown.

In $\mathcal S$ the corresponding ``subspace of complete lines'' is $\mathcal C\times [0,1]\cup [0,1]\times\mathcal C$ where $\mathcal C$ is the middle-thirds Cantor set. This space is known to have conformal dimension exactly $Q=1+\log(2)/\log(3)$ by work of Bishop--Tyson \cite{BiTy} and 
$$\frac{\log(2)}{\log(6)} = \frac{Q - 1}{Q}.$$ 
It is a very recent result of Kwapisz \cite{Kwapisz_cdim} that the conformal dimension of $\mathcal S$ is strictly greater than $1+\log(2)/\log(3)$. 

We can change many parameters to obtain a wider collection of boundaries. Let $d\geq 2$, $b\geq 3$, $A\subseteq [0,b-1]\cap\mathbb Z$ and $0\leq m\leq d$. For each $k$ define $\Gamma_k(d,b,A,m)$ to be the graph with vertex set
$$V_k = \big\{(x_1,...,x_d) \in \mathbb{Z}^d \cap [0, b^k)^{d}:\forall j \;\; \big|\{i: x_{i,j} \in A \}\big| \leq m \big\}$$
where $x_i = \sum\limits_{j=0}^{k-1} x_{i,j} \cdot b^{j}$, $A \subseteq [0,b-1]$ and $0 \leq m \leq d$. In other words, we are now working with points in $\mathbb Z^d$ such that at most $m$ coordinates have a digit from $A$ at the same position in base $b$. As before, we may build a hyperbolic graph $\Gamma(d,b,A,m)$ with spheres $\Gamma_k(d,b,A,m)$ and fractal boundary $\mathcal{S}(d,b,A,m)$. Notice that $\mathcal{S}(2,3,\{1\},1)$ is the Sierpi\'{n}ski carpet and $\mathcal{S}(3,3,\{1\},1)$ is the standard Menger sponge. The spheres $\Gamma_k(d,b,A,m)$ are again nested so we define
$$S(d, b, A, m) = \bigcup\limits_{k=0}^{\infty} \Gamma_k(d,b,A,m).$$

Our most general theorem, from which Theorem \ref{thm:sierpinski} follows immediately is the following.

\begin{restatable}{theorem}{general}
\label{thm:general}
Let $N = \sum\limits_{i=0}^{m-1} \binom{d-1}{i} |A|^i (b-|A|)^{d-i-1}$ and $E = \frac{\log(N)}{\log(N)+\log(b)} $. Then $$sep_{S(d,A,b,m)}(n) \preceq n^E$$
Moreover, for $m=1$, this is actually a $\asymp$-equality
$$sep_{S(d,A,b,1)}(n) \asymp n^E$$
with $E = \frac{(d-1) \log(b-|A|)}{(d-1) \log(b-|A|) + \log(b)}$.
\end{restatable}

Notice that for the Menger curve situation $S(3,3,\{1\},1)$, $E=\frac{Q-1}{Q}$ where $Q= \log(12)/\log(3)$ is the conformal dimension of the copy of $\mathcal C^2\times[0,1]$ contained in the Menger curve. For the three dimensional analogue of the Sierpi\'{n}ski carpet $\mathcal S^3$ corresponding to $S(3,3,\{1\},2)$ (removing only the middle-third cube of the 27 at each iteration) we see that $E=\frac{Q-1}{Q}$ where $Q= 1+ \log(8)/\log(3)$ is the conformal dimension of the copy of $\mathcal S\times[0,1]$ in $\mathcal S^3$.

We prove that the separation profile of $S(2,3,\{1\},1)$ is $n^{\frac{\log{2}}{\log{6}}}$ in Section 3. Section 4 defines the family of generalisations of the carpet and gives the proof of the more general theorem.

\subsection{Acknowledgements}
The first author was supported by the Summer Projects Scheme of Merton College, Oxford. The second author was supported by an Undergraduate Summer Project Bursary from the University of Oxford.
The authors would like to thank David Hume for introducing them to the topic and this particular problem, for substantial contributions to the introduction of this paper and for many helpful discussions in general.

\section{Separation of the Sierpi\'{n}ski carpet}

\sierpinski*

Here $\frac{\log{2}}{\log{6}} \approx 0.38685$. The theorem will follow from Corollaries \ref{sierpinski_lower} and \ref{sierpinski_upper} below.

\subsection{Lower bound}
For the lower bound, it is sufficient to exhibit a sequence of increasing subgraphs with cut $\succeq n^{\frac{\log{2}}{\log{6}}}$.

Define vertical and horizontal lines as $V_x = \{x\} \times \mathbb{Z}$ and $H_y = \mathbb{Z} \times \{y\}$ respectively. We will say that a line $L$ (either vertical of horizontal) is \textit{complete} in $\Gamma$ if $\Gamma \cap L$ is connected.

Lemma 1 shows that the complete lines of $\Gamma_k$ form subgraphs with the right property. See Figure \ref{fig:complete} for an example of one such subgraph.

\begin{figure}[h!]
\centering
\includegraphics[width=0.7\textwidth]{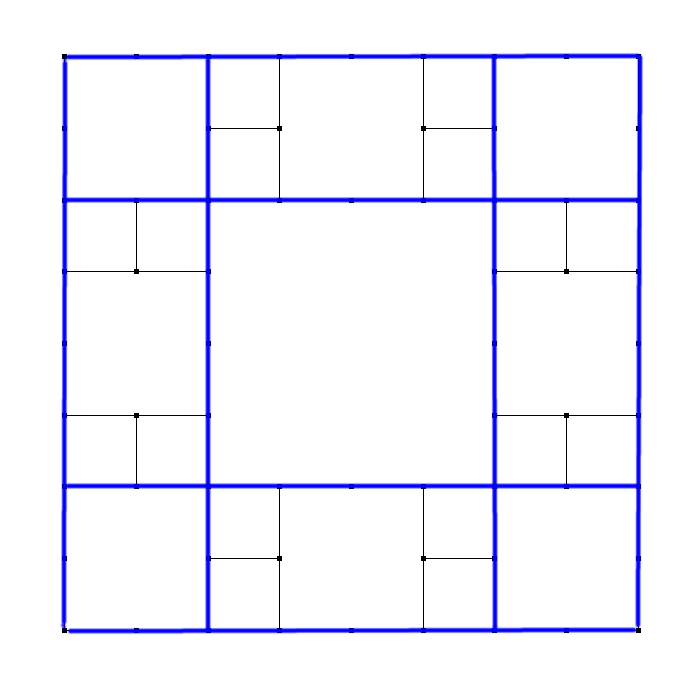}
\caption{The complete lines subgraph of $\Gamma_2$}
\label{fig:complete}
\end{figure}

\begin{lemma}
\label{complete_lemma}
Let $C_k$ be the subgraph of all complete lines in $\Gamma_k$. Then
$$ cut(C_{k}) \geq 2^{k-1} $$
\end{lemma}
\begin{proof}
$V_x$ is complete in $\Gamma_k$ if and only if $x=\sum_{i=0}^{k-1} x_i3^i$ where each $x_i\in\{0,2\}$ - otherwise, over the distance of $3^k$, there would be at least one point whose $y$-coordinate has 1 at the same position as $x$. This gives a total of $2^k$ complete horizontal lines, and the same holds for vertical lines by symmetry. Take $T \subseteq V(C_k)$ with $|T| < 2^{k-1}$. In $C_k-T$, more than half of both vertical and horizontal lines are still complete since no vertices were removed from them. But together they form a connected component of size more than $\frac{|C_k|}{2}$. Hence $cut(C_k) \geq 2^{k-1}$.

\end{proof}

\begin{corollary}
\label{sierpinski_lower}
$$ sep_S(n) \succeq n^{\frac{\log{2}}{\log{6}}} $$
\end{corollary}
\begin{proof}
Just note that $|C_k| = 2 \cdot 6^k - 4^k \asymp 6^k$.
\end{proof}

\noindent There is an alternative proof of this fact that is more technical but extends more easily to the general case. First, we'll need the following proposition.

\begin{prop}
\label{prop:paths}
Let $\Gamma$ be a connected graph on $n$ vertices with $cut(\Gamma) \leq \frac{n}{4}$. $\forall i, j \in \Gamma$ define a path $P_{i, j}$ from $i$ to $j$. Let $m = \max\limits_{v \in \Gamma} ({|\{(i, j): v \in P_{i, j}\}|}) $. Then 
$$cut(\Gamma) \geq \frac{n^2}{8m}$$
\end{prop}

\begin{proof}
Let $C$ be a cutset with $\leq \frac{n}{4}$ vertices. The components of $G-C$ can be split into two disjoint sets $A$ and $B$ each containing $\geq \frac{n}{4}$ vertices (simply by the greedy algorithm, after ordering the components in decreasing order). Since $C$ separates $A$ and $B$, all paths from $A$ to $B$ must pass through $C$. There are $2|A||B|$ (counting both directions) such paths on the one hand, and fewer than $m|C|$ on the other. Hence
$$ |C| \geq \frac{2|A||B|}{m} \geq \frac{n^2}{8m}$$
\end{proof}

\begin{lemma}
\label{complete_lemma_alt}
Let $C_k$ be the subgraph of all complete lines in $\Gamma_k$. Then
$$ cut(C_{k}) \geq \frac{2^{k}}{160} $$
\end{lemma}

\begin{proof}
It is certainly the case that $cut(C_k) \leq |C_k|/4$ - for example, just cut it with a straight line through the middle. So Proposition 1 is applicable, and we only need to define and count the paths.

Call $x \in \mathbb{Z}$ \textit{1-free} if $x$ doesn't contain 1 as a digit in base 3. As in the proof of Lemma 1, $H_x$ or $V_x$ is complete in $\Gamma_k$ if and only if $x$ is 1-free. Hence, for any $(x, y) \in C_k$, at least one of $x$ and $y$ must be 1-free.

Take $i = (x_1, y_1)$, $j = (x_2, y_2)$. By symmetry, it is sufficient to consider two cases.
\vspace{0.25cm}\\
\textit{Case 1.} $x_1$ and $y_2$ are 1-free. Define $P_{i,j}$ as follows:

$$(x_1,y_1) \xrightarrow{[1]} (x_1,y_2) \xrightarrow{[2]} (x_2,y_2)$$
\\
\textit{Case 2.} $x_1$ and $x_2$ are 1-free. Let $P_{i,j}$ be
$$(x_1,y_1) \xrightarrow{[1]} (x_1,x_1) \xrightarrow{[2]} (x_2,x_1) \xrightarrow{[3]} (x_2,y_2)$$

Here the intermediate paths corresponding to each arrow are the most direct ones along a complete line which is possible because the fixed coordinate is 1-free.

Now, given $(x,y) \in C_k$, count the paths through it. Fix one of the two cases above and consider the paths of that type in which $(x, y)$ appears as part of one of the arrows. At least one of $x$ and $y$ is equal to one of $x_i$ or $y_i$, wlog, $x_1$. Then there are at most $6^k$ choices for $j$ and $3^k$ choices for $y_1$ giving $\leq 18^k$ paths in total. Additionally, there are 5 choices of case and arrow, times 4 for symmetry so $\leq 20 \cdot 18^k$ paths in total.

By Proposition 1,
\begin{equation*}
	\begin{split}
    	cut(C_k) & \geq \frac{|C_k|^2}{8m} \\
        & \geq \frac{36^k}{8 \cdot 20 \cdot 18^k} \\
        & = \frac{2^k}{160}
    \end{split}
\end{equation*}

\end{proof}

\subsection{Upper bound}

We'll need two lemmas.

\begin{lemma}
\label{lemma:pre-square}
Let $\Gamma$ be a subgraph of $\mathcal{S}$  on $q 6^k + r$ vertices for $0 \leq q < 6$ and $0 \leq r < 6^k$. Then there is $S \subseteq \Gamma$ such that $|S| \leq 24 \cdot 2^k$, $\Gamma - S$ has at most one connected component with more than $\frac{1}{2} |\Gamma|$ vertices which, if exists, is contained in a square of side length $3^k$.
\end{lemma}

\begin{proof}

\begin{figure}[h!]
\centering
\includegraphics[width=0.7\textwidth]{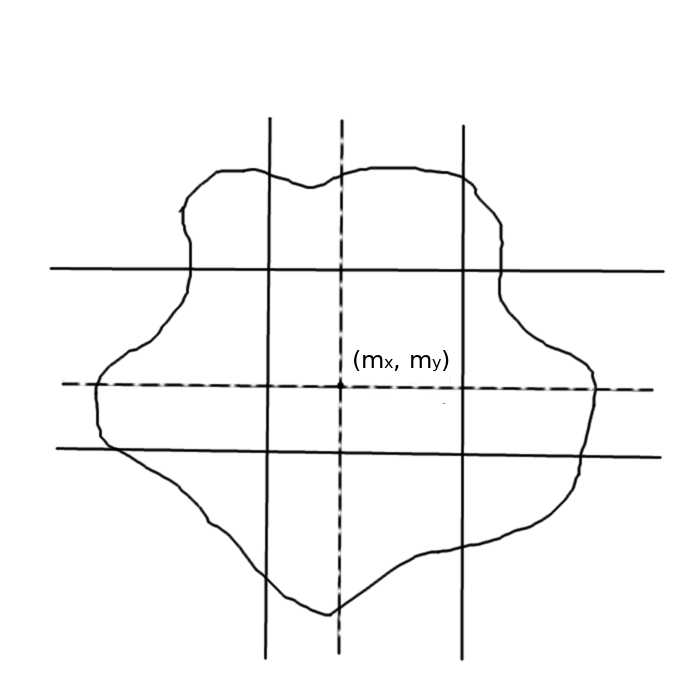}
\caption{Cutting an arbitrary subgraph of $\mathbb{Z}^2$ with straight lines: $(m_x, m_y)$ is the midpoint so that there are at most $n/2$ vertices either side of each dashed line. Thus, only the middle rectangular component can be too large.}
\label{fig:sketch}
\end{figure}

Take $m_x$ to be an integer such that $|\{(x,y) \in \Gamma : x < m_x\}| \leq \frac{1}{2}|\Gamma|$ and $|\{(x,y) \in \Gamma : x > m_x\}| \leq \frac{1}{2}|\Gamma|$ - this is either the median of the $x$-coordinates of vertices in $\Gamma$ or one of the two integers closest to it. Similarly, take $m_y$ to be such a median value with respect to the $y$-coordinates. By the pigeonhole principle, there is $d_1 \in [0,\frac{3^k}{2}]$ such that $|V_{m_x + d_1} \cap \Gamma| \leq 6 \cdot 2^k$. Similarly, there are $d_2, d_3, d_4 \in [0,\frac{3^k}{2}]$ such that each of $V_{m_x - d_2}, H_{m_y + d_3}, H_{m_y - d_4}$ has $\leq 6 \cdot 2^k$ vertices in $\Gamma$. Take $S$ to be the union of these 4 lines. By the choice of $(m_x, m_y)$, $S$ cuts $\Gamma$ into components of which only the middle one can be too large (see Figure \ref{fig:sketch}), and this component lies within a square of side length $3^k$.
\end{proof}

Thus, we can restrict our attention to cutting subgraphs of squares of side length $3^k$. The following lemma shows that this can always be done with about $2^k$ vertices.

\begin{lemma}
\label{lemma:square}
Given $k \in \mathbb{Z}_{\geq 0}$, let $Q_k$ denote a square in $\mathbb{Z}^2$ of side length $3^k$. Then for any subgraph $\Gamma$ of $S \cap Q_k$,
$$cut(\Gamma) \leq 12\cdot2^{k}$$
\end{lemma}

\begin{proof}
The proof is by induction. The statement holds trivially for $k=0$. Now suppose it holds for some $k$. Let $\Gamma \leq S$ be a subgraph lying within some $Q_{k+1}$. Choose a midpoint $(m_x,m_y)$ of $\Gamma$ as in Lemma \ref{lemma:pre-square} and note that for any $p_1, p_2, q_1, q_2$ such that $p_1 \leq m_x \leq p_2$ and $q_1 \leq m_x \leq q_2$, removing $H_{p_1}$, $H_{p_2}$, $V_{q_1}$ and $V_{q_2}$ disconnects $\Gamma$ into components, of which only the middle one can be too large (see Figure \ref{fig:sketch}). We would like these lines to intersect at most $c 2^{k+1}$ vertices of $\Gamma$ (for a constant $c$) and make the middle component a subgraph of $Q_k$ to apply the induction hypothesis.
\newline

\begin{claim}
There exist $p_1$, $p_2$ such that $p_1\leq m_x \leq p_2$, $p_2 - p_1 = 3^k$ and $$|V_{p_i} \cap \Gamma| \leq 3 \cdot 2^{k}$$ 
\end{claim}
\begin{claimproof}
Consider $p$ of the form $p = \lambda 3^{k} + \sum\limits_{i=0}^{k-1} 3^{i}$ for $\lambda \not\in (0,1)$, i.e. the last k digits (base 3) of $p$ are all 1. Then $(p,y) \in S$ only if the last $k$ digits of $y$ are either 0 or 2. There are at most $3 \cdot 2^{k}$ choices for $y$ within $Q_{k+1}$, which gives $|V_p \cap \Gamma| \leq 3 \cdot 2^{k}$. Take $p_1$ to be the largest such $p$ with $p_1 \leq m_x$ and take $p_2 = p_1 + 3^k$. Note that $p_2$ has the required form and $p_2 \geq m_x$ by the maximality of $p_1$. The claim follows.
\end{claimproof}
\vspace{0.5cm}

\noindent Choose $p_1$, $p_2$ as in the claim and similarly choose $q_1$, $q_2$ for vertical cutting lines. Remove the 4 lines $H_{p_i}$ and $V_{q_i}$. As has been noted, the resultant graph has at most one component that can be too large - but by the choice of $p_i$ and $q_i$, it has to lie within a square of side length $3^k$, which can be cut with $\leq 12\cdot2^k$ vertices by the induction hypothesis. Hence 
\begin{equation*}
    \begin{split}
        cut(\Gamma) & \leq \sum\limits_{i=1,2}\left(|V_{p_i} \cap \Gamma|+|H_{q_i} \cap \Gamma|\right) + 12\cdot2^k \\
        & \leq 4 \cdot 3 \cdot 2^k + 12 \cdot 2^k \\
        & = 12 \cdot 2^{k+1}
    \end{split}
\end{equation*}
\end{proof}

\begin{corollary}
\label{sierpinski_upper}
$$sep_S(n) \preceq n^{\frac{\log{2}}{\log{6}}}$$
\end{corollary}
\begin{proof}
Take any subgraph of $H \leq \mathcal{S}$. There exist $q$, $k$ and $r$ such that $$|H| = q6^k+r$$ where $0 \leq q < 6$ and $0 \leq r < 6^k$. Then it follows from Lemmas \ref{lemma:pre-square} and \ref{lemma:square}, that $cut(H) \leq 24 \cdot 2^k + 12 \cdot 2^k = 36 \cdot 2^k$. The corollary follows.
\end{proof}

\section{Generalisations}
In this section we consider $S(d, b, A, m)$ as defined in the introduction. Recall that the vertex set of the $k$-th iteration is now taken to be
$$V_k = \big\{(x_1,...,x_d) \in \mathbb{Z}^d \cap [0, b^k)^{d}:\forall j \;\; \big|\{i: x_{i,j} \in A \}\big| \leq m \big\}$$
where $x_i = \sum\limits_{j=0}^{k-1} x_{i,j} \cdot b^{j}$, $A \subseteq [0,b-1]$ and $0 \leq m \leq d$. See Figure \ref{fig:gens} for some (2-dimensional) examples of $\Gamma_2$ for different $b$ and $A$.

\begin{figure}[h!]
\begin{subfigure}{0.5\textwidth}
\caption{$S(2,5,\{1,3\},1)$}
\includegraphics[width=\textwidth]{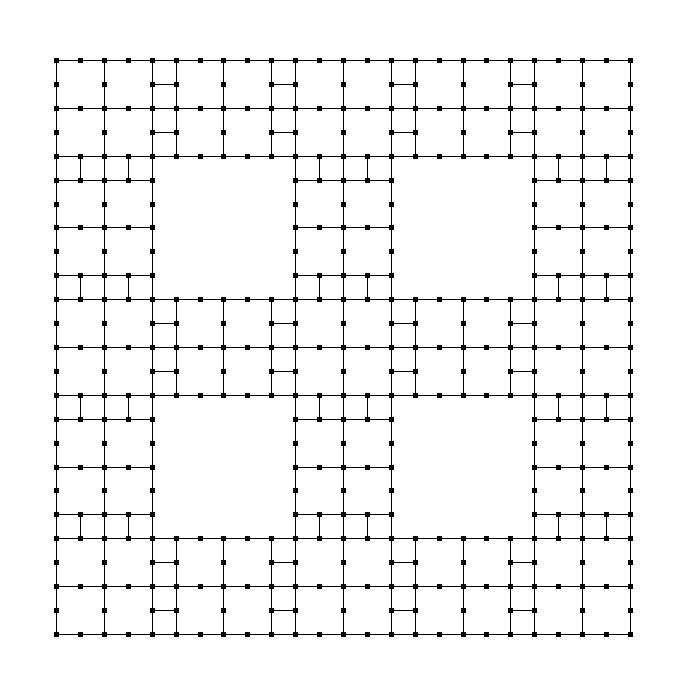}
\end{subfigure}
\begin{subfigure}{0.5\textwidth}
\caption{$S(2,5,\{0,3, 4\},1)$}
\includegraphics[width=\textwidth]
{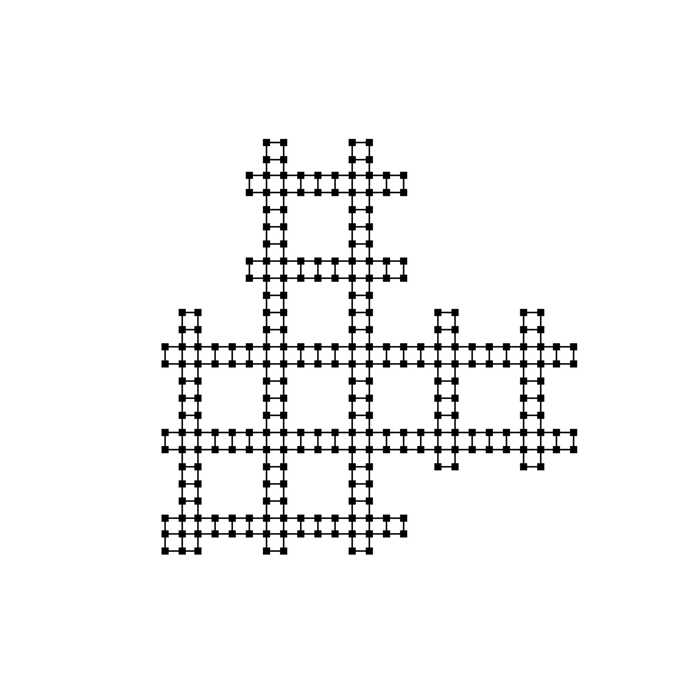}
\end{subfigure}
\begin{subfigure}{0.5\textwidth}
\caption{$S(2,7,\{1,3\},1)$}
\includegraphics[width=\textwidth]{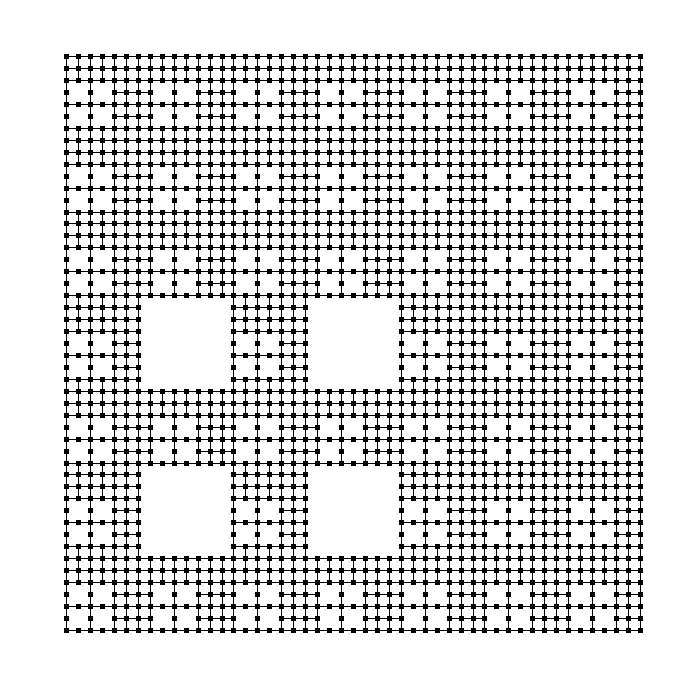}
\end{subfigure}
\begin{subfigure}{0.5\textwidth}
\caption{$S(2,8,\{1,3,6\},1)$}
\includegraphics[width=\textwidth]{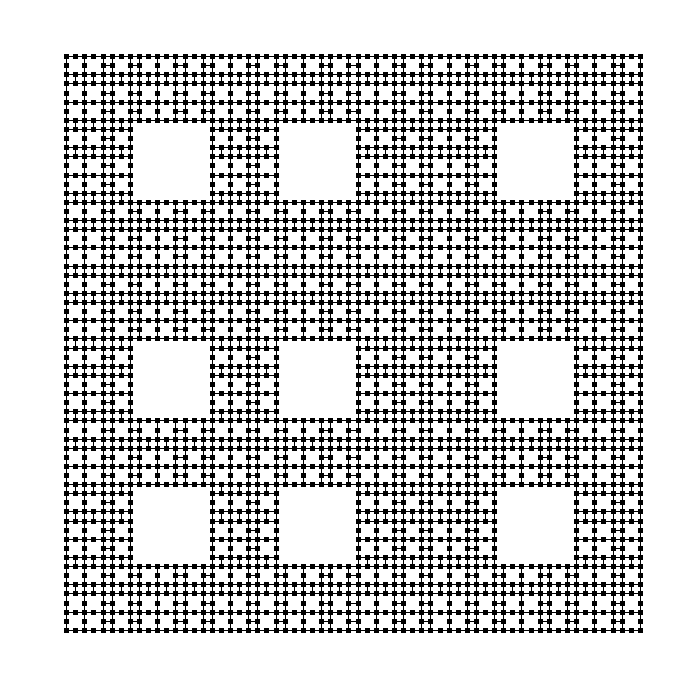}
\end{subfigure}
\caption{$\Gamma_2$ for some $S(d,b,A,m)$}
\label{fig:gens}
\end{figure}

Theorem \ref{thm:general} gives an upper bound for the separation profiles of these graphs, which is shown to be optimal in the case of $m=1$.

\general*

$N$ in the statement of the theorem may seem mysterious at first, but it is simply related to the size of the complete lines subgraphs of $S(d, b, A, m)$. To be precise, if we count the number of complete lines in the $k$-th iteration, we find that it's $\asymp N^k$, as proved in the following proposition:

\begin{prop}
\label{prop:complete_num}
Let $C_k$ be the complete lines subgraph of $\Gamma_k$ for $S(d, b, A, m)$. Then $|C_k| \asymp (N b)^k$.
\end{prop}

\begin{proof}
Since we are working up to a constant and with subgraphs of $\mathbb{Z}^d$, it is sufficient to count the lines in one direction, for example, $e_1 = (1, 0, \ldots , 0)$. Then a line defined by fixing $x_2, \ldots , x_d$ is complete if and only if at most $m-1$ of these coordinates have a digit from $A$ in any given position.

Writing the coordinates one under another in base $b$, we have
\begin{gather*}
x_{2,1} x_{2,2} \ldots x_{2,k}\\
x_{3,1} x_{3,2} \ldots x_{3,k}\\
\ldots \\
x_{d,1} x_{d,2} \ldots x_{d,k}\\
\end{gather*}
Each column is allowed at most $m-1$ $A$-digits which gives
$$N = \sum\limits_{i=0}^{m-1} \binom{d-1}{i} |A|^i (b-|A|)^{d-i-1}$$
possible configurations, and since there are $k$ columns, the number of complete lines is $N^k$. Finally, each line is of length $b^k$ giving $|C_k| \asymp (N b)^k$.
\end{proof}

If one considers intersections of $\Gamma_k$ or $C_k$ with $(d-1)$-planes perpendicular to a vector in the standard basis of $\mathbb{Z}^d$, say, $e_1 = (1, 0, \ldots, 0)$, then $N^k$ is also the size of the smallest intersection. This is because such a plane intersects fewest vertices when all digits of $x_1$ are in $A$, so the points in the plane are precisely those with at most $m-1$ coordinates having an $A$-digit in the same position.

Intuitively, it would make sense that removing it is the ``best'' way of separating $C_k$, and if we could show this, we would also have the lower bound giving a stronger theorem with equality for all cases. We believe that this is indeed the case but the lower bound proofs for the carpet are much harder to extend for $m>1$.

Here we present a proof for $m=1$, generalising that of Lemma \ref{complete_lemma_alt}.

\begin{lemma}
\label{complete_lemma_m1}
$sep_{S(d,b,A,1)}(n) \succeq n^E$ where $E = \frac{(d-1)\log(b-|A|)}{(d-1)\log(b-|A|)+\log(b)}$
\end{lemma}
\begin{proof}
Consider the subgraph $C_k$ of complete lines in $\Gamma_k$, as before. A line in direction $e_i$ (a unit vector with 1 as its $i$-th coordinate) is complete if and only if all coordinates, except the $i$-th one, are $A$-free (in other words, they don't contain a digit from $A$). Here, $N = (b-|A|)^{d-1}$, and Proposition \ref{prop:complete_num} gives $|C_k| \asymp b^{k} (b-|A|)^{(d-1)k}$.

We'll prove that $cut(C_k) \succeq (b-|A|)^{(d-1)k}$.
Take $x, y \in C_k$ with $x = (x_1, \ldots, x_d)$ and $y = (y_1, \ldots, y_d)$. Since $x$ lies on a complete line, at most one of its coordinates contains a digit from $A$, say $x_{i_1}$ (if there is no such coordinate take $i_1 = 1$) and let $x_{i_2}$ be any other coordinate. Similarly define $y_{j_1}$. Then we can build $P_{x,y}$ in essentially the same way as for Lemma \ref{complete_lemma_alt}, moving along a complete line at every step:

\begin{enumerate}
\item Replace the $i_1$-th coordinate of $x$ with $x_{i_2}$.

\item Replace every coordinate of $x$ with the corresponding coordinate of $y$, skipping the $i_1$-th and ${j_1}$-th coordinates. Note that at the end of every replacement, all coordinates are $A$-free so that we are free to move in any direction.

\item Replace the $i_1$-th coordinate with $y_{i_1}$. If $i_1 = j_1$, we are done, otherwise do the next step.

\item Replace the $j_1$-th coordinate with $y_{j_1}$
\end{enumerate}

At every stage at least $d-1$ of an intermediate point's coordinates are the same as those of $x$ or $y$ and additionally $A$-free.
Now, as before, given a point $z$, count the number of possible $P_{x,y}$ containing it. At least $d-1$ of $z$'s coordinates are shared with either $x$ or $y$, but there are $d+1$ more coordinates to specify and at most 2 of them can contain an $A$-digit. This gives $(b-|A|)^{(d-1)k} b^{2k}$ choices. By Proposition \ref{prop:paths},
$$ cut(C_k) \geq \frac{|C_k|^2}{8(b-|A|)^{(d-1)k} b^{2k}} = \frac{(b-A)^{(d-1)k}}{8}$$
as required.
\end{proof}

By contrast, the upper bound proof extends to all cases in a straight-forward way.

\begin{lemma}
\label{gen_upper}
Let $N = \sum\limits_{i=0}^{m-1} \binom{d-1}{i} |A|^i (b-|A|)^{d-i-1}$ and $E = \frac{\log(N)}{\log(N)+\log(b)} $. Then $$sep_{S(d,A,b,m)}(n) \preceq n^E$$
\end{lemma}

\begin{proof}
Let $\Gamma$ be a subgraph of $S(d,A,b,m)$ on $q (b N)^k + r$ vertices for some $0 \leq q < bN$ and $0 \leq r < (bN)^k$. We define the midpoint of $\Gamma$ in a similar way to Lemma \ref{lemma:pre-square}. Exactly as with the Sierpi\'{n}ski carpet, $\Gamma$ can be cut by removing $(d-1)$-planes of $\preceq N^k$ vertices around the midpoint so that only the component containing the midpoint can be too large, and it will lie in a $d$-cube of side length $b^k$ (as can be arranged by the pigeonhole principle, using $d$ planes with at most $b N^{k+1}$ vertices each). Now it remains to show that the cut of subgraphs in such a cube is of order $N^k$, and the upper bound will follow as before.

Fix $i$ and consider $(d-1)$-planes given by
$$x_i = \lambda b^k + \sum\limits_{i=0}^{k-1} a_i b^{i}$$
for $a_i \in A$. The size of these planes is at most $\frac{b^{d-1}}{(b-|A|)^{d-1}} N^k$ (almost as in the sparsest plane, but one digit may be allowed to be in $A$) and the distance between any two is at most $b^k$. So, exactly as before, given a subgraph of a $d$-cube of side length $b^{k+1}$, find its midpoint $(m_1, \ldots, m_d)$ and for each $i$ take the largest $x_i \leq m_i$ of the form as above and let $y_i = x_i + b^k$ noting that $y_i \geq m_i$. Removing the planes corresponding to ${x_i, y_i}$ cuts the $d$-cube with at most $\frac{2 d b^{d-1}}{(b-|A|)^{d-1}} N^k$ vertices, which is a constant multiple of $N^k$ as required.

This completes the proof.

\end{proof}

Combining this with Lemma \ref{gen_upper} immediately gives Theorem \ref{thm:general}.


\begin{thebibliography}{HMT17}

\bibitem[BS11]{BonkSchramm}
M.~Bonk and O.~Schramm.
\newblock Embeddings of {G}romov hyperbolic spaces.
\newblock In {\em Selected works of {O}ded {S}chramm. {V}olume 1, 2}, Sel.
  Works Probab. Stat., pages 243--284. Springer, New York, 2011.
\newblock With a correction by Bonk.

\bibitem[BST12]{BenSchTim}
I.~Benjamini, O.~Schramm, and {\'A}.~Tim{\'a}r.
\newblock On the separation profile of infinite graphs.
\newblock {\em Groups Geom. Dyn.}, 6(4):639--658, 2012.

\bibitem[BT01]{BiTy}
C.~Bishop and T.~Tyson.
\newblock Locally minimal sets for conformal dimension.
\newblock {\em Ann. Acad. Sci. Fenn. Math.}, 26(2):361--373, 2001.

\bibitem[DGP11]{DGP_split}
F.~Dahmani, V.~Guirardel and P.Przytycki.
\newblock Random groups do not split.
\newblock {\em Math. Ann.}, 349(3):657--673, 2011.

\bibitem[Gr03]{Gr_rand}
M.~Gromov.
\newblock Random walk in random groups.
\newblock {\em Geom. Funct. Anal.}, 13(1):147--177, 2003.

\bibitem[Hei01]{Hein}
J.~Heinonen.
\newblock {\em Lectures on analysis on metric spaces}.
\newblock Universitext. Springer-Verlag, New York, 2001.

\bibitem[HK98]{HK}
J.~Heinonen and P.~Koskela.
\newblock Quasiconformal maps in metric spaces with controlled geometry.
\newblock {\em Acta Math.}, 181(1):1--61, 1998.

\bibitem[HMT17]{HMT_Poincareprofiles}
D.~Hume, J.~Mackay, and R.~Tessera.
\newblock Poincar\'e profiles of groups and spaces.
\newblock Available from arXiv:1707:02151, 2017.

\bibitem[Kwa16]{Kwapisz_cdim}
J.~Kwapisz.
\newblock {Conformal Dimension via $p$-Resistance: Sierpi\'{n}ski Carpet}.
\newblock {\em Invent. Math.}, 205(1):173--220, 2016.

\bibitem[LT79]{LipTar}
R.~J.~Lipton and R.~E.~Tarjan.
\newblock A separator theorem for planar graphs.
\newblock {\em SIAM J. Appl. Math.}, 36(2):177--189, 1979.

\bibitem[Pan89]{Pansu_confdim}
P.~Pansu.
\newblock Dimension conforme et sph\`{e}re \`{a} l'infini des vari\'{e}t\'{e}s
  \`{a} courbure n\'{e}gative.
\newblock {\em Ann. Acad. Sci. Fenn. Ser. A I Math.}, 14(2):177--212, 1989.

\end{thebibliography}
\end{document}